\documentclass[a4paper,draft]{amsproc}
\usepackage{amssymb}
\usepackage{amscd}

\theoremstyle{plain}
 \newtheorem{thm}{Theorem}[section]

 \newtheorem{cor}{Corollary}[section]
\theoremstyle{definition}
 
 \newtheorem{dfn}{Definition}[section]
\theoremstyle{remark}
 \newtheorem{rem}{Remark}[section] 
 \numberwithin{equation}{section}

\renewcommand{\leq}{\leqslant}
\renewcommand{\geq}{\geqslant}

\begin{document}

\title[A tour through $\delta$-invariants]{A tour through $\delta$-invariants: From Nash's embedding theorem to  ideal immersions, best ways of living and beyond$^{\,1}$}



\author[Chen]{\bfseries Bang-Yen Chen}

\address{
Department of Mathematics \\ 
Michigan State University   \\ 
619 Red Cedar Road\\ East Lansing, Michigan 48824--1027\\
U.S.A.}
\email{bychen@math.msu.edu}

\thanks{{${}^{1)}$\,For general references of this article, see \cite{c7,c2000,book}.}} 

\dedicatory{To the memory of my esteemed friend and coworker,\\Franki Dillen}

\vspace{18mm}
\setcounter{page}{1}
\thispagestyle{empty}

\begin{abstract}
First I will explain my motivation to introduce the $\delta$-invariants for Riemannian manifolds. I will also recall the notions of ideal immersions and best ways of living. Then I will present a few  of the many applications of $\delta$-invariants to several areas in mathematics. Finally, I will present two optimal inequalities involving $\delta$-invariants for Lagrangian submanifolds obtained very recently in joint papers with F. Dillen, J. Van der Veken and L. Vrancken. 
\end{abstract}

\maketitle

\section{Why  introduced $\delta$-invariants\,?}

In this section I will explain my motivation to introduce the $\delta$-invariants.

\subsection{Nash's embedding theorem}  

In 1956, John F. Nash (1928\,-- )  proved in \cite{nash} the following famous embedding theorem. 

\begin{thm} Every  Riemannian $n$-manifold can be isometrically 
 embedded in a Euclidean $m$-space $\mathbb E^m$ with $m=\frac{n}{2} (n+1)(3n+11)$.
 \end{thm}
 
 For example, Nash's theorem implies that every 3-dimensional Riemannian manifold can be isometrically embedded in $\mathbb E^{120}$ with codimension 117.
 
 The embedding problem had been around since Bernhard Riemann (1826--1866) and was posed explicitly by Ludwig Schl\"afli (1814-1895) in the 1870s. The problem  has evolved naturally from a progression of other questions that had been
posed and partly answered beginning in the middle of 19th century via the work of Jean F. Frenet (1816-1900), Joseph A. Serret (1819-1885), Louis S.\,X. Aoust (1814-1885); and then by Maurice Janet (1888-1984) and \'Elie Cartan (1869--1951) in 1920s.
  First, mathematicians studied the embedding problem of ordinary curves, then of surfaces, and finally, of Riemannian manifolds of higher dimensions.
 
John Conway (1937-- ),  the Princeton mathematician who discovered surreal numbers and invented the game of life  (known simply as  Life), called Nash's result  
\vskip.05in

``{\it one of the most important pieces of mathematical 
analysis in the 20th

\  century\,}''.
 
\noindent Also,  according to Shlomo Sternberg (1936-- ) of Harvard University,
\vskip.05in

``{\it The embedding problem is a deep philosophical question concerning the 

\ \ foundations of geometry that virtually every mathematician\,--\,from 

\ \ Riemann and  Hilbert  to \'Elie Cartan and Hermann Weyl\,--\,working  in 
 
 \ \  the field  of differential geometry for the past century has asked himself.}''
 \vskip.05in

\noindent  (Quoted from the book: ``A Beautiful Mind; The life of mathematical genius and 
  Nobel laureate John Nash'' by Sylvia Nasar, 1998.)
 
\subsection{Why is it so difficult to apply Nash's embedding theorem\,?}

 Nash's embedding theorem was aimed for in the hope that if Riemannian  manifolds could always be regarded as Riemannian submanifolds, this would then yield the opportunity to use extrinsic
help.  Till when observed in  \cite{gromov2} as such by Mikhail L. Gromov (1943-- ), this hope had not been materialized however. 

There were several reasons why it is so difficult to apply Nash's theorem. One reason is that it requires very large codimension  for a Riemannian manifold to admit an isometric embedding in Euclidean spaces in general, as stated in Nash's theorem. On the other hand, submanifolds of higher codimension are very difficult to understand, e.g., there are no general results {for arbitrary Riemannian submanifolds}, except the three fundamental equations of Gauss, Codazzi and Ricci.

Another reason, for instance, was explained in  \cite{Y} as  follows.
\vskip.05in

{\it ``\,What is lacking in the Nash theorem is the control of the extrinsic 

\ \ quantities in relation to the intrinsic quantities.''}

\vskip.05in
 In other words, another  main difficulty to apply Nash's  theorem is: 
 
 \vskip.05in
``{\it There  do not exist general optimal relationships between the known 

\ \ intrinsic  invariants and the main extrinsic invariants for arbitrary 
   
   \ \ Riemannian submanifolds of Euclidean spaces.}''
 \vskip.05in

\subsection{Obstructions to minimal immersions}
Since there are no obstructions to isometric embeddings according to Nash's  theorem, in order to study  isometric embedding (or more generally, immersion) problems,   it is natural
    to impose  some suitable conditions on immersions. 
 
  For example, if one  imposes the {\it minimality condition} it leads to the following problem proposed by Shiing-Shen Chern (1911-2004) during the 1960s.
  \vskip.05in

{\bf Problem 1}: {\it What are necessary  conditions for a Riemannian manifold to admit a minimal isometric immersion into a Euclidean space\,?}
  \vskip.05in

The equation of Gauss states that the Riemann curvature tensor $R$ and the second fundamental form $h$ of a Riemannian submanifold  in a Euclidean space satisfy
 $$R(X,Y;Z,W)=\left<h(X,W),h(Y,Z)\right>-\left<h(X,Z),h(Y,W)\right>.$$
It follows immediately from the  equation of Gauss that a necessary condition for a Riemannian manifold to admit a minimal immersion in any Euclidean space  is 
   $$ Ric\,\leq 0,\;\; \text{(in particular,  $\tau\leq 0$)},$$
   where $Ric$ and $\tau$ are the Ricci and scalar curvatures of the submanifold.

{\it For many  years}, before the invention of $\delta$-invariants, {\it this was  the only known  Riemannian obstruction} for a general Riemannian manifold to admit a minimal immersion into a Euclidean space with arbitrary codimension.

\subsection{Obstructions to Lagrangian immersions}

A submanifold $M$ of a K\"ahler manifold $(\tilde M,J,g)$ is called {\it Lagrangian} if $$J(T_pM)=T^\perp_p M,\;\; \forall p\in M,$$
where $J$ and $g$ are the complex structure and the K\"ahler metric of $\tilde M$.
   
  A result of  Gromov states that a compact $n$-manifold $M$ admits a Lagrangian immersion into the complex Euclidean $n$-space $\hbox{\bf C}^n$ if and
  only if the complexification $TM\otimes \hbox{\bf C}$ of the tangent bundle of $M$ is trivial  \cite{gromov}.
  Gromov's result implies that there are {\it no topological obstructions} 
  to Lagrangian immersions for compact 3-manifolds in $\hbox{\bf C}^3$, because the 
  tangent bundle of every 3-manifold is trivial. 

 On the other hand, if one imposes isometrical condition to the Lagrangian immersion problem, it leads to the following.
\vskip.1in

{\bf {Problem 2}:} {\it
What are the Riemannian obstructions to  Lagrangian isometric immersions of Riemannian $n$-manifolds into ${\bf C}^n$?}

\subsection{What do we need to do to overcome the difficulties\,?}

We ask the following two basic questions:
    \vskip.05in
    
(1)   {\it  What do we need to do to overcome the difficulties mentioned above\,?} 

\hskip.23in  (in order to apply Nash's embedding theorem)
  \vskip.05in
  
 (2) {\it How can we solve Problems 1 and 2 concerning minimal and Lagrangian 
 
 \ \ \ \ \ \   immersions\,?}
  \vskip.05in

\subsection{My answers} My answers to these two basic questions are the following.
 
 \vskip.05in
 (a)  {\it We need to  introduce new type of Riemannian invariants}, {\bf  different in   
 
\hskip.25in  nature}, {\it from  the  ``\,classical\,''  Riemannian invariants; namely, from 

\hskip.25in scalar and Ricci  curvatures} 
   (which have been studied extensively
   
    \hskip.25in  for more than 150 years since Riemann).
 \vskip.05in
 
 (b) {\it We also need to establish general optimal relationships between the main 

\hskip.25in extrinsic  invariants of the submanifolds and the new type of intrinsic

\hskip.25in  invariants. }

\vskip.1in
\subsection{My motivation} These considerations provided me the motivation to introduce $\delta$-invariants for Riemannian manifolds during the 1990s.

\section{How I defined $\delta$-invariants}

Let $M$ be a  Riemannian $n$-manifold and $K(\pi)$ denotes the sectional curvature of a plane section $\pi\subset T_pM,\, p\in M.$
   For an orthonormal basis $\{e_1,\ldots,e_n\}$ of $T_pM$, the {\it scalar curvature $\tau$} at $p$ is given by
 \begin{equation}\label{2.1}\tau(p)=\sum_{i<j} K(e_i\wedge e_j).\end{equation}

 Let $L$ be a subspace of $T_pM$  of dimension $r\geq 2$  and let $\{e_1,\ldots,e_r\}$ be an 
  orthonormal basis of $L$. The {\it scalar curvature $\tau(L)$} of $L$ is defined by    \begin{equation} \label{2.2}\tau(L)=\sum_{\alpha<\beta} K(e_\alpha\wedge e_\beta),\quad 1\leq\alpha<\beta\leq r.\end{equation}
  
\subsection{The set ${\mathcal S}(n)$} For a given integer $n\geq 2$,  denote by ${\mathcal S}(n)$ the 
finite set consisting of all $k$-tuples $(n_1,\ldots,n_k)$ of integers $2\leq n_1,\ldots,n_k\leq n-1$ with $n_1+\cdots+n_k\leq n$,  where $k$ is a non-negative integer.

In number theory and combinatorics, a partition of a positive integer $n$ is a way of writing $n$ as a sum of positive integers.
The  cardinal number $\#\mathcal S(n)$ of $\mathcal S(n)$ is equal to $p(n)-1$, where $p(n)$ is the number of partitions of $n$.  For instance, for
\begin{equation}\begin{aligned}\notag n=\, &2,3,4,5,6,7,8,9,10,
\ldots,
\\& 50,\ldots,100,\ldots,200,\end{aligned}\end{equation}
the corresponding cardinal numbers are given respectively
by
\begin{equation}\begin{aligned}\notag &\hskip.3in \#\mathcal S(n)=1,2,4,6,10,14,21,29,41,\ldots, \\&\hskip.1in  204\, 225,\ldots, 190\,569\,291,\ldots,3\,972\,999\,029\,387.\end{aligned}\end{equation}
 The asymptotic behavior of $\#\mathcal S(n)$ is given by
$$\#{\mathcal S}(n)\approx \text{\small$\frac{1}{4n\sqrt{3}}\exp \left[\sqrt{\tfrac{2n}{3}}\,\pi \right]$}\quad \hbox{as}\;\;n\to\infty.$$

\subsection{Definition of $\delta$-invariants}
For each given $k$-tuple $(n_1,\ldots,n_k)\in {\mathcal S}(n)$,  I defined the $\delta$-invariant  $\delta(n_1,\ldots,n_k)$  by
\begin{align}\label{2.3} & \delta(n_1,\ldots,n_k)=\tau- \inf\{\tau(L_1)+\cdots+\tau(L_k)\},\end{align} where
 $L_1,\ldots,L_k$ run over all $k$ mutually orthogonal subspaces of $T_pM$ such that  $\dim L_j=n_j,\, j=1,\ldots,k$. 
In particular, we have
\vskip.05in 

$\delta(\emptyset)=\tau$ \ \   ($k=0$, the trivial $\delta$-invariant),
\vskip.05in

$\delta{(2)}=\tau-\inf K,$ where $K$ is the sectional curvature,

\vskip.05in
$\delta(n-1)(p)=\max Ric(p)$.
\vskip.1in

\begin{rem}  \label{R:2.1}The non-trivial $\delta$-invariants are {\it very different in nature} from the ``classical'' scalar and Ricci curvatures; simply due to the fact that both scalar and Ricci curvatures are the ``total sum'' of sectional curvatures on a Riemannian manifold.
  In contrast, the non-trivial $\delta$-invariants  are obtained  from the scalar curvature 
  by throwing away a certain amount of sectional curvatures. Consequently, the non-trivial $\delta$-invariants are {\it very weak intrinsic invariants}.

 Borrowing a term from biology,  {\it $\delta$-invariants are  DNA of Rieman\-nian manifolds}. Results in later sections illustrate that all of our $\delta$-invariants do affect directly  the behavior of Riemannian manifolds. \end{rem}
 
\begin{rem} \label{R:2.2} Some other invariants of similar nature, i.e., intrinsic invariants obtained from scalar curvature by throwing away certain amount of sectional curvature, are also  called $\delta$-invariants.  For instance, we also have affine $\delta$-invariants, K\"ahlerian  $\delta$-invariants, submersion $\delta$-invariants, contact $\delta$-invariants, etc. 
Such $\delta$-invariants were introduced while we investigated some special families of manifolds.\end{rem}

\section{What can we  do with $\delta$-invariants\,?} 

To apply Nash's embedding theorem, we ask the following most fundamental question in the theory of Riemannian submanifolds.
\vskip.05in

{\bf Fundamental Question:} 
 {\it What can we conclude from an arbitrary isometric immersion of a Riemannian manifold in a Euclidean space with any codimension\,?} 
\vskip.03in

That is,   {\bf ``\,$\forall$\,arbitrary isometric immersion $\phi: M\to \mathbb E^m\implies ???$\,''}

\vskip.03in
Or, more generally,
\vskip.03in

\noindent {\it What can we conclude from an arbitrary isometric immersion between Riemannian manifolds\,?}
\vskip.03in

That is,  {\bf ``\,$\forall$\,arbitrary isometric immersion $\phi: M\to \tilde M \implies ???$\,''}
\vskip.03in

\begin{rem} For surfaces in $\mathbb E^{3}$, the equation of Gauss was found in 1827 in principal, though not explicitly, by Carl F. Gauss (1777--1855); and the equation of Codazzi was given by
  Delfino Codazzi (1824--1875) in 1860,  independently by Gaspare Mainardi (1800--1879) in 1856 and also by Karl M. Peterson (1828--1881) in his 1853 doctoral thesis [\,\"Uber die Biegung der Fl\"achen,  Dorpat University\,].    In 1880,  Aurel E. Voss (1845--1931) extended both equations to Riemannian submanifolds. The equation of Ricci was discovered by Gregorio Ricci (1853--1925)  in 1899.  

For so many years since then, the {\it only known solutions} to the Fundamental Question were the equations of Gauss, Codazzi and Ricci, as we already said earlier.
\end{rem}

\subsection{A new general solution to Fundamental Question} By applying the $\delta$-invariants, we are able to provide the following new optimal general solution to the Fundamental Question.

\begin{thm}\label{T:3.1} For  any isometric immersion of a  Riemannian $n$-manifold $M$ 
 into another  Riemannian manifold $\tilde M$,   we have 
\begin{equation}\begin{aligned}\label{3.1} & \delta{(n_1,\ldots,n_k)} \leq  \frac{n^2(n+k-1-\sum_{j=1}^k n_j)}{2(n+k-\sum_{j=1}^k n_j)} H^2 \\&\hskip.6in +\frac{1}{2} \text{\small$ \Bigg\{$}{{n(n-1)}}-\sum_{j=1}^k {n_j(n_j-1)} \text{\small$ \Bigg\}$}\max\tilde K\end{aligned}\end{equation}
 for every $k$-tuple $(n_1,\ldots,n_k)$ $\in \mathcal S(n)$, where $H^{2}$ denotes the squared mean curvature and $\max\tilde K(p)$ is the maximum of  sectional curvatures of $\tilde M$ restricted to 2-plane sections of $T_pM$.
  
  The equality case of inequality \eqref{3.1} holds at a point $p\in M$ if and only if the following two conditions hold:
\begin{itemize}
 
  \item[{\rm (a)}]  there exists an  orthonormal basis  $\{e_1,\ldots,e_m\}$ of $\,T_{p}M$ such that  the shape operator $A$ at $p$ takes the form:
\begin{align}\label{13.30}\font\b=cmr10 scaled \magstep1 \def\bigzerol{\smash{\hbox{ 0}}} \def\bigzerou{\smash{\lower.0ex\hbox{\b 0}}} A_{e_r}=\text{\small$\left( \begin{matrix} A^r_{1} & \hdots & 0 \\ \vdots  & \ddots& \vdots &\bigzerou \\ 0 &\hdots &A^r_k& \\&\bigzerou & &\mu_rI \end{matrix} \right)$},\quad  r=n+1,\ldots,m, \end{align}
where $I$ is an identity matrix and  $A^r_j$ is a symmetric $n_j\times n_j$  submatrix satisfying
$\hbox{\rm trace}\,(A^r_1)=\cdots=\hbox{\rm trace}\,(A^r_k)=\mu_r.$

\item[{\rm (b)}] for any $k$ mutual orthogonal subspaces $L_1,\ldots,L_k$ of $T_pM$ satisfying
$$\delta(n_1,\ldots,n_k)=\tau-\text{\small$\sum$}_{j=1}^k \tau (L_j)$$ at $p$, we have
$ \tilde K(e_{\alpha_i},e_{\alpha_j})=\max \tilde K(p)$
 for any $\alpha_i\in \Delta_i,\alpha_j\in \Delta_j$ with $1\leq i\ne j\leq k+1$, where $\Delta_{1},\ldots,\Delta_{k+1}$ are given by
 \begin{equation}\begin{aligned}\notag &\Delta_1=\{1,\ldots, n_1\},\; \ldots \\&\Delta_{k}=\{n_1\!+ \cdots+\! n_{k-1}\!+\!1,\ldots, n_1\!+\cdots+\! n_k\},\\& \Delta_{k+1}=\{n_1\!+\cdots+\!n_{k}\!+\!1,\ldots, n\}.\end{aligned}\end{equation}
\end{itemize}
\end{thm}

\begin{dfn}\label{D:3.1}For each $k$-tuple $(n_1,\ldots,n_k)\in \mathcal S(n)$, we define the {\it normalized $\delta$-invariant}
 $\Delta(n_1,\ldots,n_k)$ of a Riemannian $n$-manifold by
 \begin{equation}\label{3.3}\Delta(n_1,\ldots,n_k)=
\frac{\delta(n_1,\ldots,n_k)}{c(n_1,\ldots,n_k)},\; \; c(n_1,\ldots,n_k)= \frac{n^2(n\!+\!k\!-\!1\!-\!\sum_{j=1}^k n_j)}{2(n+k-\sum_{j=1}^k n_j)}.\end{equation}\end{dfn}

When the ambient space $\tilde M$ is a Euclidean space, Theorem \ref{T:3.1} becomes

\begin{thm}\label{T:3.2} For any isometric immersion of a Riemannian $n$-manifold $M$ into a Euclidean space with arbitrary codimension, we have
     \begin{align}\label{Ideal2}& H^2 \geq \Delta(n_1,\ldots,n_k)\end{align}   
for every $(n_1,\ldots,n_k)\in \mathcal S(n)$.  

  The equality case of  \eqref{Ideal2} holds at  $p\in M$ if and only if 
 there is an  orthonormal basis  $\{e_1,\ldots,e_m\}$ of $T_{p}M$ such that  the shape operator $A$ at $p$ takes the form:
\begin{align}\label{3,5}\font\b=cmr10 scaled \magstep1 \def\bigzerol{\smash{\hbox{ 0}}} \def\bigzerou{\smash{\lower.0ex\hbox{\b 0}}} A_{e_r}=\text{\small$\left( \begin{matrix} A^r_{1} & \hdots & 0 \\ \vdots  & \ddots& \vdots &\bigzerou \\ 0 &\hdots &A^r_k& \\&\bigzerou & &\mu_rI \end{matrix} \right)$},\quad  r=n+1,\ldots,m, \end{align}
where $I$ is an identity matrix and  $A^r_j$ is a symmetric $n_j\times n_j$  submatrix which satisfy 
$\hbox{\rm trace}\,(A^r_j)=\mu_r$ for $j=1,\ldots,k$.

\end{thm}

\begin{rem} For each $(n_1,\ldots,n_k)\in \mathcal S(n)$, inequality \eqref{Ideal2} is optimal, since  there exists a non-minimal submanifold satisfying the equality case. \end{rem}

\begin{rem} Due to the well-known fact that the mean curvature vector field  of an isometric immersion is exactly the tension field, Theorem \ref{T:3.2} shows that the amount of tension a submanifold  receives at each point from its ambient space  is {\it predominated below by its $\delta$-invariants.} Consequently, each $\delta$-invariant does affect directly the behavior of Riemannian manifolds.\end{rem}

Theorem  \ref{T:3.2} implies immediately the following.

\begin{cor} For every isometric immersion  of a Riemannian $n$-manifold into a Euclidean space  with arbitrary codimension, we have
\begin{align}\label{3.6} H^2(p)\geq \hat\Delta_0(p),\end{align} 
where $\hat\Delta_{0}:=\max \{\Delta(n_1,\ldots,n_k): (n_1,\ldots,n_k)\in {\mathcal S}(n)\} $ is the  maximal normalized $\delta$-invariant.
\end{cor}

\subsection{Maximum principle}

In general, there do not exist direct relationships between  $\delta$-invariants.  On the other hand, we have the following maximum principle, which follows immediately from Theorem \ref{T:3.2}.

\vskip.1in
\noindent {\bf Maximum Principle.} {\it If an $n$-dimensional submanifold $M$ of a Euclidean space satisfies the equality case of \eqref{Ideal2} for some $k$-tuple $(n_1,\ldots,n_k)\in {\mathcal S}(n)$, i.e.,
\begin{align} \label{3.7} H^2=\Delta(n_1,\ldots,n_k),\end{align} 
 then for every $(m_1,\ldots,m_j)\in {\mathcal S}(n)$ and any $j$ we have}
\begin{align} \label{3.89}\Delta(n_1,\ldots,n_k)=\hat\Delta_{0}\geq\Delta(m_1,\ldots,m_j).\end{align}

\section{Ideal immersions and best ways of living} 

 What is a ``{\bf nice immersion}'' of a Riemannian manifold\,? In my opinion, it is an isometric immersion which produces the least possible amount of tension at each point. For this reason I introduced the notion of ideal immersions in the 1990s.
 
\begin{dfn}\label{D:4.1} An isometric immersion of a Riemannian $n$-manifold into a Euclidean space is called {\it ideal\/} if $H^2= \hat\Delta_0$ holds identically. \end{dfn}
 
 The  Maximum Principle yields the following important fact.

\begin{thm} \label{T:4.1} If an isometric immersion of a Riemannian $n$-manifold $M$ in a Euclidean space satisfies $H^2=\Delta(n_1,\ldots,n_k)$ identically for a $k$-tuple $(n_1,\ldots,n_k)$ in $ {\mathcal S}(n)$, then it is an ideal immersion automatically. \end{thm}

\begin{rem} \label{R:4.1} Theorem \ref{T:3.2} implies that ideal submanifolds are those which receive the {\it least amount of tension at each point}.  \end{rem}

In the following, by a {\it best world} we mean a  Riemannian space with the highest degree of homogeneity.

 \begin{rem} \label{R:4.2} According to the work of Sophus Lie (1842-1899), Felix  Klein (1849-1925) and Wilhelm Killing (1847-1923), the family of best worlds consists of Euclidean spaces, spheres, real projective spaces, and real hyperbolic
spaces, i.e., the family of real space forms. These spaces have the highest degree of homogeneity, since their groups of isometries have the maximal possible dimension.  \end{rem}
 
 \begin{dfn}\label{D:4.2}  A {\it best way of living} is  an ideal isometric  embedding of a Riemannian manifold into a best world. \end{dfn}

\begin{rem} \label{R:4.3}It follows from Definition \ref{D:4.2} that a best way of living is a very comfortable way of living in a wonderful world which allows maximal degree of freedom (living in a best world), it preserves  shape (isometric, no stretching), without self-cutting (embedding). Furthermore it receives the least possible amount of tension at each point from the outside world (ideal).
\end{rem}

\section{Some applications of $\delta$-invariants}  

In this section we provide some of the many applications of $\delta$-invariants.  Many more applications can be found in my 2011 book \cite{book}.

\subsection{Applications to minimal immersions} $\delta$-invariants have many applications to minimal immersions.
As an immediate application of Theorem \ref{T:3.2}, we have the following solution to Problem 1 concerning Riemannian obstructions to minimal immersions.  
  
  \begin{thm}\label{T:5.1}  Let $M$ be a Riemannian $n$-manifold. If there exist a point $p$ and a $k$-tuple $(n_1,\ldots,n_k)\in {\mathcal S}(n)$ with $\delta(n_1,\ldots,n_k)(p)>0$,  
   then $M$  never admits a minimal immersion into  any Riemannian  manifold
     with non-positive sectional curvature. 
In particular, $M$ never admits  a minimal immersion into any Euclidean space for any codimension.\end{thm}

\subsection{Applications to spectral theory}

By applying Theorem \ref{T:3.2} and Nash's embedding theorem, we discovered the following intrinsic result. 
 
 \begin{thm} \label{T:5.2} Let $M$ be a compact irreducible homogeneous Riemannian $n$-manifold. Then the first nonzero eigenvalue  $\lambda_1$ of  the
 Laplacian $\Delta$  satisfies
  \begin{align}\label{5.1}  \lambda_1\geq n \,  \Delta(n_1,\ldots,n_k)\end{align}
for every $k$-tuple $(n_1,\ldots,n_k)\in \mathcal S(n)$. 
\end{thm}
 
 If $k=0$, then \eqref{5.1} reduces to the following well-known result of Tadashi Nagano (1930\,-- ) obtained in \cite{nagano}:
 \begin{align}\label{5.2} \lambda_1\geq n\rho,\; \text{where$\;\,\rho=\frac{2\tau} {n(n-1)}\;\,$is the {\it normalized  scalar  curvature}}.\end{align}

Theorem \ref{T:5.2} implies immediately the following.

\begin{cor}\label{C:5.1}  For every compact irreducible homogeneous Riemannian $n$-manifold,
 we have $\lambda_{1} \geq n \hat\Delta_{0}.$ \end{cor}

\subsection{Who can live in the wonderland of best livings\hskip.004in?}

 Ordinary spheres in Euclidean spaces are the simplest examples of best ways of living.  Besides spheres there are many other homogeneous spaces which admit best ways of living in some Euclidean spaces. 
For instance, the following three compact homogeneous spaces:
\begin{align}\notag SU(3)/T^2,\;\; Sp(3)/Sp(1)^3,\;\; \mbox{and}\;\;  F_4/\mbox{Spin}(8)\end{align} admit best ways of living in $\mathbb  E^8,\, \mathbb E^{14}$ and $\mathbb E^{26}$ associated with $$(3,3)\in{\mathcal S}(6),\; (3,3,3,3)\in {\mathcal S}(12),\;\; (12,12)\in{\mathcal S}(24),$$ respectively. 
The best ways of living for $SU(3)/T^2,\,Sp(3)/Sp(1)^3$ and $\,F_4/\mbox{Spin}(8)$ in $\mathbb E^8,\,\mathbb E^{14}$ and $\mathbb E^{26}$ are given respectively by the corresponding minimal isoparametric hypersurfaces in $S^7,S^{13}$ and $S^{25}$.

 Not every Riemannian manifold is lucky enough to admit a best way of living in a Euclidean space. 
     For example,  although the unit $n$-sphere $S^{n}(1)$ admits  a best way of living in $\mathbb E^{n+1}$; the real projective $n$-space $RP^{n}(1), n>1,$ never admits a best way of living in any Euclidean space no matter how hard he/she tries (see Corollary \ref{C:5.2}). 
     
     Consequently, I proposed in  \cite{c7} the following two problems: 
  \vskip.05in

    {\bf Existence Problem}: {\it What  are the necessary and sufficient conditions  for 

\hskip.2in a  Riemannian manifold  to admit  a best way of living in a 
  best world?}
   
\vskip.1in

 {\bf Classification Problem}: {\it If a Riemannian manifold admits best ways of
   
\hskip.2in    living in a best world, what are his/her best ways of living?}

\subsection{Two solutions to  Existence Problem}

 By applying  $\lambda_{1}\geq n\hat \Delta_{0}$ given in Corollary \ref{C:5.1} and  a result from the theory of finite type submanifolds, we have the following solution for the Existence Problem.
   \vskip.05in

\begin{thm} \label{T:5.3} A compact irreducible homogeneous Riemannian $n$-manifold 
admits a best way of living in some Euclidean space if and only if it satisfies the intrinsic condition: $ \lambda_{1}=n\hat \Delta_{0}.$ \end{thm}

\begin{cor}\label{C:5.2}  $RP^n(1)$ admits no best ways of living in any Euclidean space.\end{cor}
\begin{proof} For $RP^{n}(1)$ we have  $ \lambda_{1}=2(n+1)$ and $\hat\Delta_{0}=1$. Thus $\lambda_{1}\ne n \hat \Delta_{0}$.
\end{proof}

Another simple solution of the Existence Problem is the following.

\begin{thm}\label{T:5.4}   If a compact Riemannian $n$-manifold $M$ satisfies
\begin{align} \lambda_1> \frac{n}{v(M)}\int_M\hat\Delta_0*1,\;\;\; v(M)=\text{\rm volume of $M$},\end{align}
then it never admits a best way of living in any Euclidean space.
\end{thm}

\subsection{Applications to rigidity theory: Uniqueness problem}
 If a Riemannian manifold admits a best way of living in a best world and if it is unique,  it gives a rigidity theorem. 
 Otherwise, it admits multiple ways of best  living (lucky one\,!).
 
  For instance, by applying $\delta$-invariants we proved the following rigidity results.
 
 \begin{thm}\label{T:5.5} If $M$ is an open portion of $S^n(1)$, then for every isometric immersion of $M$ in a Euclidean $m$-space with arbitrary codimension we have
\begin{align}\label{5.4} H^2\geq 1.\end{align}
The equality case of  \eqref{5.4} holds identically if and only if $M$ is embedded as an open portion of an ordinary hypersphere  in a totally geodesic $\mathbb E^{n+1}\subset \mathbb E^m$.\end{thm}

  \begin{thm} \label{T:5.6}  Let $M$ be an open part of  ${\mathbb E}^{n_1-1}\times S^{n-n_1+1}(1)$. Then for every isometric immersion of $M$  in $\mathbb E^{m}$ with arbitrary codimension we have 
\begin{align}\label{5.5} H^2\geq \left(\text{\small$ \frac {n-n_1+1}{n}$}\right)^{\! 2}\!.\end{align}

The equality case of \eqref{5.5} holds if and only if  $M$ is immersed as an open part of  a spherical hypercylinder ${\mathbb E}^{n_1-1}\times S^{n-n_1+1}(1) \subset \mathbb E^{n_1-1}\times \mathbb E^{n-n_1+2}\subset \mathbb E^m$. \end{thm}

  \begin{rem}  Rigidity results discovered in this way are quite different from usual rigidity results, e.g. rigidity in the sense of  Aleksandr D. Aleksandrov (1912--1999).  Because rigidity theorems via ideal immersions do not require any assumption on topology or codimension of the submanifolds.   
   On  the other hand, (almost) all other rigidity results do require global conditions as well as assumptions of very small codimension (e.g. in the theory of convex surfaces \`a la Aleksandrov). \end{rem}

\subsection{Applications to warped products}
 An application of $\delta$-invariants to warped products is to obtain the following sharp result. 

\begin{thm} \label{T:5.7} For every isometric immersion $\phi:N_1\times_f N_2\to \tilde M$ of a 
warped product $N_1\times_f N_2$ into any Riemannian manifold $\tilde M$,  the warping function satisfies
\begin{align}\label{14.37}&  \frac{\Delta f}{f}\leq \frac{(n_1+n_2)^2}{4n_2} H^2+ n_{1} \max \tilde K,\;\; n_i=\dim N_i,\; i=1,2,\end{align} where $\Delta$ is the Laplacian on $N_1$.
\end{thm}
 
 This theorem has many immediate consequences. E.g., it yields the following.
 
 \begin{cor}\label{C:5.3} Let $N_1\times_f N_2$ be a warped product of Riemannian manifolds. If the warping function $f$ is harmonic, then we have:

\begin{itemize}
  \item[{\rm (1)}]    $N_1\times_f N_2$ never admits minimal immersion into any  Riemannian manifold of negative sectional curvature; 

\item[{\rm (2)}]   every minimal immersion of $N_1\times_f N_2$ into any Euclidean space is a warped product immersion regardless of codimension. \end{itemize}\end{cor}

\begin{cor}\label{C:5.4} If $f\in C^{\infty}(N_{1})$ is an eigenfunction of $\Delta$ with eigenvalue $\lambda>0$, then every warped product $N_1\times_f N_2$ never admits a
minimal immersion into any Riemannian manifold with non-positive sectional curvature.
\end{cor}

\begin{cor}\label{C:5.5} Let  $N_1$ be a compact Riemannian manifold. Then we have:
\begin{itemize}
\item[{\rm (1)}]  every warped product $N_1\times_f N_2$ never admits a minimal immersion into any Riemannian manifold of negative sectional curvature;
    
\item[{\rm (2)}]  every  warped product $N_1\times_f N_2$ never admits any minimal immersion into any Euclidean space.
\end{itemize}
\end{cor}

 \subsection{Applications to theory of submersions}

A Riemannian submersion $\pi: M\to B$ is called {\it trivial} if it is a direct product of a fiber and the base manifold $B$.
Two applications of $\delta$-invariants to the theory of submersions are the following.
   
  \begin{thm}\label{T:5.8}  If a Riemannian manifold $M$ admits a non-trivial Riemannian  submersion with totally geodesic fibers,  it cannot be isometrically  immersed  into any Riemannian manifold  of non-positive sectional curvature as a  minimal submanifold.
\end{thm}

\begin{thm}\label{T:5.9}  If a Riemannian manifold $M$ admits a Riemannian  submersion with totally geodesic fibers, then every minimal immersion of $M$ in a Euclidean is the direct product of a minimal immersion of the base manifold and a minimal immersion of a fiber in Euclidean spaces.
\end{thm}

\subsection{Applications to affine geometry} $\delta$-invariants can also be applied to affine geometry. For instance, by applying $\delta$-invariants we have the following.

 \begin{thm}  \label{T:5.10} If the Calabi metric of an improper affine hypersphere in an affine space  is the Riemannian product metric of $k$ Riemannian manifolds,  then the improper affine hypersphere  is the Calabi composition of $k$ improper affine spheres.
 \end{thm} 

\begin{thm} \label{T:5.11}   If the warping function $f$ of a warped product manifold  $N_1\times_f N_2$ satisfies $\Delta f < 0$  at some point on $N_1$, then $N_1\times_f N_2$ cannot be realized as an improper affine  hypersphere in an affine $(n+1)$-space ${\bf R}^{n+1}$. \end{thm}

 \begin{thm} \label{T:5.12}  Every  warped product  $N_1\times_f N_2$ with harmonic  warping function cannot be realized as an elliptic  proper affine hypersphere in ${\bf R}^{n+1}$. \end{thm}

\subsection{Links between submersions and affine hypersurfaces} By applying an affine $\delta$-invariant, we discovered the following links between theory of Riemannian submersions and affine differential geometry.

 \begin{thm} \label{T:5.13}  If  a Riemannian manifold $M$ can be realized as an elliptic proper centroaffine hypersphere centered at the origin  in some affine space, then every Riemannian submersion $\pi:M\to B$   with minimal fibers has non-totally geodesic horizontal distribution.\end{thm}

 \begin{thm} \label{T:5.14}  If a Riemannian manifold $M$ can be realized as an improper hypersphere in an affine space, then every Riemannian submersion $\pi:M\to B$   with non-totally geodesic minimal fibers has non-totally geodesic horizontal distribution. \end{thm}

 \subsection{\!Applications to symplectic geometry}
  
 An application of $\delta$-invariants to Lagrangian submanifolds is to provide a  sharp solution to Problem 2.

  \begin{thm}\label{T:5.15} Let $M$ be a compact Riemannian $n$-manifold with null first Betti number $b_{1}(M)$ or finite fundamental group $\pi_{1}(M)$. If there exists a $k$-tuple $(n_1,\ldots,n_k)\in {\mathcal S}(n)$   such that 
$ \delta{(n_1,\ldots,n_k)}>0,$
then $M$ never admits a Lagrangian isometric immersion into $\mbox{\bf C}^n$. \end{thm}

\begin{rem} The assumption  on $\delta{(n_1,\ldots,n_k)}$ in Theorem \ref{T:5.15} is sharp. This can be seen as follows:
 Consider  the {\it Whitney  sphere} $W^n$ defined by  the Whitney immersion $w:S^n\to {\bf C}^n$: 
 $$w(y_{0},\ldots,y_{n})=\text{$\frac{1+{\rm i} y_{0}}{1+y_{0}^{2}}$}(y_{1},\ldots,y_{n})  $$
 with $y_{0}^{2}+y_{1}^{2}+\cdots+y_{n}^{2}=1.$
  This immersion  is Lagrangian with a unique self-intersection point at $w(-1,0,\ldots,0)$ $=w(1,0,\ldots,0)$. 
For each $k$-tuple ${(n_1,\ldots,n_k)}$, we have $\delta{(n_1,\ldots,n_k)}\geq 0$  with respect to the induced metric via $w$. Moreover, we have $\delta{(n_1,\ldots,n_k)}=0$ only at the unique point of self-intersection.
Furthermore, the assumption of  $\#\pi_1(M)<\infty$ or  $b_1(M)=0$ is necessary for $n\geq 3$ (see \cite{book}).
\end{rem}

The proof of Theorem \ref{T:5.15} is based on the next result.
 
 \begin{thm}\label{T:5.16} Let $M$ be a Lagrangian submanifold of a complex
 space form $M^n(4c)$. 
Then for each $k$-tuple $(n_1,\ldots,n_k)\in \mathcal S(n)$ we have
\begin{equation}\begin{aligned}\label{L1}& \delta{(n_1,\ldots,n_k)} \leq  \frac{n^2\big(n+k-1-\sum_{j=1}^k n_j\big)}{2\big(n+k-\sum_{j=1}^k n_j \big)} H^2 \\& \hskip1.3in +\text{$\frac{1}{2}$}\text{\small$\Bigg\{$}{{n(n-1)}}-\sum_{j=1}^k {n_j(n_j-1)} \text{\small$\Bigg\}$} c\end{aligned}\end{equation}
 
 \end{thm}

We have the following result from \cite{c00b} for  Lagrangian submanifolds satisfying the equality of  \eqref{L1};  extending a result of \cite{cdvv96} on $\delta(2)$-ideal Lagrangian submanifolds.
 
\begin{thm}\label{T:5.17} If a Lagrangian submanifold of a complex space form  satisfies 
  the equality case of \eqref{L1} at a point,   then it is  minimal at that point.
\end{thm}

\section{Two recent optimal inequalities involving $\delta$-invariants for Lagrangian submanifolds}

In view of Theorem \ref{T:5.17}, it is very natural to look for optimal inequalities for non-minimal Lagrangian submanifolds in complex space forms.

\subsection{Case I: $\sum_{i=1}^{k}n_{i}<n$}
Recently, joint with F. Dillen, J. Van der Veken and L. Vrancken, we improve inequality
 \eqref{L1} in \cite{CD11,cdvv} to the following inequality for $\sum_{i=1}^{k}n_{i}<n$ (see also \cite{CD11.2}).

\begin{thm} \label{T:6.1} If $M$ is a Lagrangian submanifold of a complex space form $M^n(4c)$, then for a $k$-tuple  $(n_1,\ldots,n_k)\in {\mathcal S}(n)$ with $\sum_{i=1}^{k}n_{i}<n$ we have 
\begin{equation}\begin{aligned} \label{L2} &\hskip-.0in \delta(n_1,\ldots,n_k) \leq \, \frac{n^2 \big\{ n- \sum_{i=1}^{k}n_{i} + 3k-1\! -6\, {\sum_{i=1}^k} (2+ n_{i})^{-1} \big\}}{2 \big\{n-\sum_{i=1}^{k}n_{i}+3k+ 2-6\,{\sum_{i=1}^k} (2+ n_{i})^{-1} \big\}} H^2\
\\&\hskip1.6in  +\frac{1}{2}\text{\small$\Bigg\{$}n(n-1)-\sum_{i=1}^k n_i(n_i-1)\text{\small$\Bigg\}$}c\,.\end{aligned}\end{equation}
The equality sign of \eqref{L2} holds at a point $p\in M$ if and only if there exists an orthonormal basis $\{e_1,\ldots,e_n\}$ at $p$ such that with respect to this basis the second fundamental form $h$ takes the following form:
\begin{equation}\begin{aligned} \label{15.10}&h(e_{\alpha_i},e_{\beta_i})=\sum_{\gamma_i} h^{\gamma_i}_{\alpha_i \beta_i} Je_{\gamma_i}+\text{\small$\frac{3\delta_{\alpha_i\beta_i} }{2+n_i}$}\lambda Je_{\mu+1},\;  \sum_{\alpha_i=1}^{n_i} h^{\gamma_i}_{\alpha_i\alpha_i}=0,
\\& h(e_{\alpha_i},e_{\alpha_j})=0,\; i\ne j, \;
h(e_{\alpha_i},e_{\mu+1})=\text{\small$ \frac{3\lambda}{2+n_i}$} J e_{\alpha_i},\;   h(e_{\alpha_i},e_u)=0,
\\& h(e_{\mu+1},e_{\mu+1})=3\lambda Je_{\mu+1},
\; h(e_{\mu+1},e_u)=\lambda Je_u ,
\; h(e_u,e_v)=\lambda \delta_{uv} Je_{\mu+1},\;\;\end{aligned}\end{equation} for $1\leq  i,j\leq k; \mu+2\leq u,v\leq n $ and $ \lambda=\frac{1}{3}h^{\mu+1}_{\mu+1 \mu+1}$, where $\mu=n_1+\cdots+n_k$.
\end{thm}

\begin{rem} For $\delta(2)$, inequality \eqref{L2} is due to T. Oprea [\,Chen's inequality in the Lagrangian case, Colloq. Math. {\bf 108} (2007),  163--169\,].\end{rem}

In \cite{cdvv} we also proved the following.
 
 \begin{thm} \label{T:6.2} For every $k$-tuple $(n_1,\ldots,n_k)\in \mathcal S(n)$, there exists a non-minimal Lagrangian submanifold which satisfies the equality case of \eqref{L2}.\end{thm}

Theorem \ref{T:6.2} shows that inequality \eqref{L2}  cannot be improved further.

\subsection{Case II: $\sum_{i=1}^{k}n_{i}=n$}

For this case we prove the following inequality with different coefficient in \cite{cdvv}. 
 
  \begin{thm} \label{T:6.3} Let $M$ be a Lagrangian submanifold of a complex space form $M^n(4c)$. Then for each
 $(n_1,\ldots,n_k)\in {\mathcal S}(n)$ with $\sum_{i=1}^{k}n_{i}=n$ we have
\begin{equation}\begin{aligned} \label{L3} & 
\delta(n_1,\ldots,n_k) \leq  \frac{n^2 \big\{k-1-2\, {\sum_{i=2}^k} (2+ n_{i})^{-1}\big\}}{2 \big\{k-2\,{\sum_{i=2}^k} (2+ n_{i})^{-1} \big\}} H^2
\\&\hskip1.3in  +\frac{1}{2}\text{\small$ \Bigg\{$}n(n-1)-\sum_{i=1}^k n_i(n_i-1)\text{\small$\Bigg\}$} c\, ,\end{aligned}\end{equation}
where we assume that $n_1=\min_{i=1}^{n}\{ n_i\}$.

If the equality sign of \eqref{L3} holds at a point $p\in M$, the components of the second fundamental form with respect to some suitable orthonormal basis $\{e_{1},\ldots,e_{n}\}$ for $T_{p}M$ satisfy the following conditions:
\begin{itemize}
\item[{\rm (a)}] $h^{A}_{\alpha_{i}\alpha_{j}}=0$ for $i\ne j$ and $A\ne \alpha_{i},\alpha_{j}$;

\item[{\rm (b)}]  if $n_j \neq \min\{n_1,\ldots,n_k\}$: 
$$h_{\alpha_i \alpha_i}^{\beta_j}=0 \mbox{ if } i \neq j \mbox{ and } \sum_{\alpha_j \in \Delta_j} h_{\alpha_j \alpha_j}^{\beta_j} = 0,$$
\item[{\rm (c)}] if $n_j = \min\{n_1,\ldots,n_k\}$: 
$$\sum_{\alpha_j \in \Delta_j} h_{\alpha_j \alpha_j}^{\beta_j} = (n_i+2) h_{\alpha_i \alpha_i}^{\beta_j} \mbox{ for any } i \neq j \mbox{ and any } \alpha_i \in \Delta_i.$$

\end{itemize}
\end{thm}

\begin{rem}
In the case of equality, we don't have information about $h_{\alpha_i \beta_i}^{\gamma_i}$, where $\alpha_i$, $\beta_i$ and $\gamma_i$ are mutually different indices in the same block $\Delta_i$.
\end{rem}

 The following theorem from \cite{cdvv} implies that  inequality \eqref{L3} is sharp.  
\begin{thm} For each $k$-tuple $(n_{1}.\ldots,n_{k})\in {\mathcal S}(n)$ with $\sum_{i=1}^{k}n_{i}=n$, there
exists a Lagrangian submanifold in $M^{n}(4c)$ satisfying the equality of the improved inequality \eqref{L3} identically. \end{thm}
 
 This result implies that inequality \eqref{L3} cannot be improved further as well.

 \begin{rem} Lagrangian submanifolds of complex space forms satisfying some special cases of the improved inequalities  \eqref{L2} and \eqref{L3} have been investigated and classified recently in \cite{bo,BV,book,CD11,cdvv,cdvv2,cdv,cp,cpw}.
 \end{rem}

\bibliographystyle{amsplain}

\end{document}